\makeatletter
\documentclass[runningheads]{amsart}

\relax
\relax
\relax
\relax
\usepackage{graphicx}
\usepackage{amsmath}
\usepackage{amsfonts}
\usepackage[dvipsnames]{xcolor}
\usepackage{subcaption}
\usepackage{graphicx}
\usepackage{xspace}
\usepackage{amssymb}
\usepackage{algorithm}
\usepackage{algpseudocode}
\usepackage{mathtools}
\usepackage{array}
\usepackage{ragged2e}
\usepackage{float}
\usepackage{setspace}

\definecolor{linkcolor}{rgb}{0.65,0,0}
\definecolor{citecolor}{rgb}{0,0.4,0}
\definecolor{urlcolor}{rgb}{0,0,0.65}
\usepackage[backref=page,colorlinks=true,linkcolor=linkcolor,urlcolor=urlcolor,citecolor=citecolor,bookmarksnumbered=true,bookmarksopen=true,bookmarksopenlevel=2,unicode]{hyperref}

\usepackage{cleveref}

\usepackage{tikz}
\usetikzlibrary{arrows.meta,bending,positioning}

\newtheorem{theorem}{Theorem}
\newtheorem{lemma}[theorem]{Lemma}

\newtheorem{definition}{Definition}

\newtheorem{conjecture}[theorem]{Conjecture}
\newtheorem{remark}[theorem]{Remark}

\newcommand{\epsp}{\mathrm{epsp}}

\newcommand{\lcm}{\mathrm{lcm}}

\newcommand{\Z}{\mathbb{Z}}

\newcolumntype{R}[1]{>{\RaggedLeft\arraybackslash}m{#1}}

\title[Ours goes to 211: Euler Pseudoprimes to 47 Prime Bases]{Ours goes to 211: Euler Pseudoprimes to 47 Prime Bases (from Carmichael numbers)}
\author[A.~Alcantarilla Sánchez]{Alejandra Alcantarilla Sánchez}
\address{Alejandra Alcantarilla Sánchez \\
Eindhoven University of
Technology\\
Netherlands}
\email{alejandra.alcantarilla.sanchez@gmail.com}

\author[J.~Cottaar]{Jolijn Cottaar}
\address{Jolijn Cottaar \\
Eindhoven University of
Technology\\
Netherlands}
\email{jolijncottaar@gmail.com}

\author[T.~Lange]{Tanja Lange}
\address{Tanja Lange\\
Eindhoven University of
Technology\\
Netherlands\\
and Academia
Sinica\\
Taiwan}
\email{tanja@hyperelliptic.org}

\author[B.~de Weger]{Benne de Weger}
\address{Benne de Weger\\
Eindhoven University of
Technology\\
Netherlands}
\email{b.m.m.d.weger@tue.nl}

\begin{document}
\begin{abstract}
In this paper we show that a certain subset of the Carmichael numbers contains good
Euler pseudoprimes,
composite numbers that for many bases survive the Solovay-Strassen primality test.
We present a classification of Carmichael numbers,
and use the knowledge gained from this to create a fast algorithm to compute new Euler
pseudoprimes,
by multiplying already found Euler pseudoprimes.
We use this algorithm to find many Euler pseudoprimes that are pseudoprimes for several
consecutive prime bases starting at 2,
hence for all integer bases up to that number.
The best Euler pseudoprime we find survives up to 211,
i.e.,
survives the first 47 prime bases.
\end{abstract}

\maketitle

\section{Introduction}

Many public-key cryptosystems like RSA and Diffie-Hellman need prime numbers to function
and stay secure.
The creation of these large prime numbers is usually done by sampling numbers at random
until a prime is found.
Proving that such a random number is prime is computationally intensive and typically
skipped.

Instead implementations rely on probabilistic compositeness tests,
like Solovay-Strassen or Miller-Rabin.
These tests check for randomly chosen bases that some property,
which always holds for prime numbers,
holds for the tested number.
A composite number that survives the test is called a pseudoprime.
One runs the test with multiple bases to make it unlikely for random composite numbers
to be identified as prime.
If the ``prime'' selection is adversarially controlled,
they can fool the user into accepting a composite number if they have access to pseudoprimes
with high chance of passing these tests,
Albrecht,
Massimo,
Paterson,
and
Somorovsky~\cite{CCS:AMPS18}
introduced primality testing under adversarial conditions considering the Miller-Rabin
test.
A composite passing this test is called strong pseudoprime.
Sorenson and
Webster~\cite{2017/sorenson}
constructed a strong pseudoprime passing the first 12 prime bases.

In this paper we consider the Solovay-Strassen test,
which checks for a candidate $n$ and base $a$ if
$\left(\frac{a}{n}\right)
\equiv
a^\frac{n-1}{2}
\pmod
n$.
Surviving $n$ are called Euler pseudoprimes.
We construct an Euler pseudoprime passing the first 47 primes bases.
Miller-Rabin has at most
1/4
failure cases while Solovay-Strassen has at most
1/2,
explaining a factor of 2 in the number of bases by luck,
the rest is due to our construction.

We start with Carmichael numbers,
which are composites satisfying
$a^{n-1}\equiv
1
\pmod
n$ for all $a
\in
\Z_n^*$,
and identify a subset of these as best candidates for Euler pseudoprimes.
In
Section~\ref{ch:carmichael_numbers}
we classify the Carmichael numbers and identify a large class that attains the theoretical
maximum number of passing bases.
We prove easily testable,
necessary conditions on two Carmichael numbers for their product to be a Carmichael
number in that class.
If
$n_1$
is a Carmichael number then RSA functions unchanged with
$n_1$
in place of one of one of the primes.

In
Section~\ref{sec:construct_carmichael}
we show that multiplying already found Euler pseudoprimes with each other is a good
way to find more and better Euler pseudoprimes.
\relax
\relax
We provide our algorithm to find bigger and better Euler pseudoprimes by multiplying
Euler pseudoprimes within a certain class and give an overview of the Euler pseudoprimes
we found in
Section~\ref{ch:results}.
Finally we present an alternative approach in
Section~\ref{app:anonymous}.
\relax

\section{Background on pseudoprimes}
\relax
In this section we will introduce the types of pseudoprimes we consider in this paper
and notation used in the rest of this paper.
We start with Fermat pseudoprimes,
based on Fermat's little theorem.

\begin{definition}[Fermat pseudoprimes]
If n is an odd,
composite,
positive integer that satisfies
\[\quad a^{n-1} \equiv 1 \pmod{n}\]
where $a$ is a positive integer coprime to $n$,
then $n$ is called a Fermat pseudoprime to the base $a$,
and $a$ is called a Fermat liar,
otherwise $a$ is a Fermat witness.
\end{definition}

Contrary to the Solovay-Strassen test,
when using Fermat's test there does not always exist an $a$ for each composite $n$
(coprime to $n$) such that Fermat's test fails.
These $n$ that do not have such an $a$ are called Carmichael
numbers~\cite{Carmichael1910}.

\begin{definition}[Carmichael number]
\label{def: Carmichael number}
A Carmichael number is a composite number $n$ that satisfies for all $a$ relatively
prime to $n$ that
\[a^{n-1}\equiv1\pmod{n}.\]
\end{definition}

In other words,
it is a composite number that is a Fermat pseudoprime for all bases $a$ coprime to
$n$.
We use the extensive list of all Carmichael numbers up to
$10^{24}$
by Shallue and
Webster~\cite{2025/shallue,2025/webster}
(extending their ANTS-XVI
algorithm~\cite{2024/shallue}).

Korselt~\cite{Korselt1899}
has proven the following theorem on Carmichael numbers.
\begin{theorem}[Korselt's criterion]
A composite integer $n
>
1$ is a Carmichael number if and only if
\begin{itemize}
\item[(i)] $n$ is squarefree, and
\item[(ii)] for every prime $p$ dividing $n$, also $(p-1) \mid (n-1)$.
\end{itemize}
\label{thm: Korselt}
\end{theorem}

The main primality test we concern ourselves with in this paper is the Solovay-Strassen
test,
which is based on Euler's criterion.
\begin{theorem}[Euler's criterion]
Let $p$ be an odd prime and
$a\in
\Z_p^*$.
Then
\[\left( \frac{a}{p}\right) \equiv a^{\frac{p-1}{2}} \pmod{p}.\]
\label{thm:Euler-criterion}
\end{theorem}

We can now define Euler pseudoprimes,
which for certain $a$ still adhere to Euler's criterion,
even though they are composite.

\begin{definition}[Euler pseudoprime]
If $n$ is an odd,
composite,
positive integer that satisfies
\begin{equation}
\left(\frac{a}{n}\right) \equiv a^{\frac{n-1}{2}} \pmod{n},
\label{eq:ProofSSAssumption}
\end{equation}
where $a$ is a positive integer coprime to $n$,
then $n$ is called an Euler pseudoprime to the base $a$,
and $a$ is called an Euler liar,
otherwise $a$ is called an Euler witness.
\end{definition}

\relax

There is no analogue to Carmichael numbers for Euler's criterion.

\begin{theorem}
\label{lem:solovay-strassen}
For any odd composite number
\(
n
\),
there exists an
$a\in
\Z_n^*$
such that
\begin{equation*}
\left( \frac{a}{n} \right)  \not \equiv  a^{\frac{n-1}{2}} \pmod{n}.
\label{eq: SS non equal}
\end{equation*}
\end{theorem}
A proof can be found in
\cite{Rosen19861984}
by Rosen.
\begin{lemma}\label{lem:max-phi}
For any odd composite number
\(
n
\),
there exist at most
$\frac{1}{2}
\varphi(n)$
Euler liars.
\end{lemma}
\begin{proof}
Let $w$ be an Euler witness for $n$,
which exists by
Theorem~\ref{lem:solovay-strassen},
and let $a$ be an Euler liar.
By multiplicativity of both sides in
(\ref{eq:ProofSSAssumption}),
$aw$ is an Euler witness.
Taking products with $w$,
there is at least one witness for each liar,
so that at most half of the
$\varphi(n)$
integers coprime to $n$ are Euler liars.
\end{proof}
\relax

From this we see that each $a$ has more than
50\%
chance of identifying a composite $n$ as such,
hence,
randomly chosen $n$ will not satisfy
(\ref{eq:ProofSSAssumption})
for many consecutive prime bases.
The following observation,
along with experiments,
guided our search to considering Carmichael numbers.
By
Lemma~\ref{thm:
eulerfermat},
any Euler liar is a Fermat liar

\begin{lemma}
If $n$ is an Euler pseudoprime to the base $a$,
then $n$ is a Fermat pseudoprime to the base $a$.
\label{thm: eulerfermat}
\end{lemma}
\begin{proof}
Square both sides and observe that
$\left(
\frac{a}{p}\right)
\in
\{-1,1\}$
as
$\gcd(a,n)=1$.
\end{proof}

\section{Classification of Carmichael numbers} \label{ch:carmichael_numbers}
Carmichael numbers are good candidates for Euler pseudoprimes.
In this section we introduce a classification of Carmichael numbers that determines
the prevalence of Euler liars.
By Korselt's criterion
(Theorem~\ref{thm:
Korselt}),
Carmichael numbers are squarefree,
thus we can write any Carmichael number as $n =
\prod_{i
=1}^{k}p_i$,
where
$p_i$
are distinct odd primes and $k
\geq
3$.
Our classification needs the Carmichael lambda function.

\begin{definition}[Carmichael Lambda Function]\label{def:car_lambda}
The Carmichael lambda function
$\lambda(n)$
is defined as the smallest integer
$\ell
>1$
for which
$$a^\ell
\equiv
1
\pmod
n,$$ for all $a
\in
\mathbb{Z}_n^*$.
\end{definition}
In our case of considering Carmichael numbers,
$n$ is squarefree so that
$$\lambda(n)
=
\lcm(p_1
-1,
...
,
p_k-1).$$
It is even,
$\lambda(n)|(n-1)$
(since $n$ is a Carmichael number) and
$\lambda(n)|
\varphi(n)
=
\Pi_{i=1}^k
(p_i
-1)$.
This gives the following alternative characterization of Carmichael numbers.

\begin{lemma}\label{lem:alter_carm}
Let $n$ be composite.
Then $n$ is a Carmichael number if and only if
$\lambda(n)|
(n-1)$.
\end{lemma}
\relax

We also need the index of a Carmichael number as defined by Pinch
in~\cite{pinch2006carmichaelnumbers1018}.
\begin{definition}[Index] \label{def:index}
The index of a Carmichael number $n$ is defined as $$ i(n) =
\frac{n-1}{\lambda(n)}.
$$
\end{definition}

This permits us to define a coarse classification of Carmichael numbers.

\begin{definition}[Classification of Carmichael numbers]
Let $n$ be a Carmichael number and $i(n)$ its index.
\begin{description}
\item[Class A] If $i(n)$ is even then $n$ is in class A.
\item[Class B] If $i(n)$ is odd then $n$ is in class B.
\end{description}
\end{definition}

For class A it is easy to determine the proportion of Euler liars among
$\Z_n^*$
as we show in the following lemma.

\begin{lemma}\label{lem:class_a}
Let $n$ be a Carmichael number in class A.
Then there are
$\frac12
\varphi(n)$
Euler liars in
$\Z_n^*$
and all have Jacobi symbol 1.
\end{lemma}
\begin{proof}
For all
$a\in
\Z_n^*$
we have
$a^{\frac{n-1}{2}}=\left(a^{\lambda(n)}\right)^{i(n)/2}\equiv
1
\pmod
n$ by the definition of the Carmichael function and class A.
Hence,
exactly those $a$ with
$\left(\frac{a}{n}\right)=1$
satisfy
(\ref{eq:ProofSSAssumption})
and half of the
$\varphi(n)$
elements in
$\Z_n^*$
have Jacobi symbol 1.
\end{proof}

To handle class B we need to study the square roots of 1 in
$\Z_n^*$.
Given that $n$ has $k$ factors there are
$2^k$
square roots of 1 in
$\Z_n^*$
and the group structure matches
$\Z_2^k$,
but not all may appear as
$a^{\frac{n-1}{2}}$.
Since $i(n)$ is odd and
$a^{{n-1}}\equiv
1
\pmod
n$,
we have that
$a^{\frac{n-1}{2}}=\left(a^{\frac{\lambda(n)}{2}}\right)^{i(n)}\equiv
1
\pmod
n$ implies that
$a^{\frac{\lambda(n)}{2}}\equiv
1
\pmod
n$.
By the same argument
 and
using the group structure of
$\Z_2^k$,
we have that
$a^{\frac{n-1}{2}}=\left(a^{\frac{\lambda(n)}{2}}\right)^{i(n)}\equiv
-1
\pmod
n$ implies that
$a^{\frac{\lambda(n)}{2}}\equiv
-1
\pmod
n$.
The split of class B will be along the lines of whether $-1$ appears or not,
but we give a classification that is easier to check given the factorization of $n$
which needs some further notation.

\begin{definition}[2-adic valuation]
The $2$-adic valuation of an integer $n$ is the largest integer $N$ such that
$2^N
|
n$.
We denote this as
$v_2(n)$.
\end{definition}

By minimality of
$\lambda(n)$,
we see
$v_2(\lambda(n))
=
\max_{i
\in
\{1,...,k\}}
\{v_2(p_i-1)\}$,
hence there is at least one
$p_i$
with
$v_2(\lambda(n))
=
v_2(p_i-1)$.
With that we can give the fine-grained classification of class B.

\begin{definition}[Classification of class B]
\label{def:classB}
Let $n$ be a Carmichael number of odd index $i(n)$.
Let $h$ be the number of factors
$p_i$
with
$v_2(\lambda(n))
=
v_2(p_i-1)$.
\begin{description}
\item[Class B1] If $1\le h < k$ then $n$ is in class B1.
\item[Class B2] If $h = k$ then $n$ is in class B2.
\end{description}
\end{definition}

We first establish a fact that might be of independent interest.
\begin{lemma}\label{lem:parity}
Let $n$ be a Carmichael number such that
$v_2(\lambda(n))
=
v_2(p_i-1)$
for all $
p_i
$.
Then the parity of $i(n)$ equals that of $k$.
\end{lemma}
\begin{proof}
Let
$v=v_2(\lambda(n))=v_2(p_i-1)$
for all $i$,
i.e.,
we can write each
$p_i$
as
$1+2^vr_i$
for some odd
$r_i$.
From the definition of $n$ we get
$$n=\prod_{i=1}^k
p_i
=
\prod_{i=1}^k\left(1
+2^vr_i\right)
\equiv 1 + \sum_{i=1}^k 2^vr_i \pmod{2^{2v}}.$$
Note that
$v_2(n-1)=v_2(i(n)\lambda(n))=v_2(i(n))+v$
while
$v_2\left(\sum_{i=1}^k
2^vr_i\right)=v$
if and only if $k$ is odd.
\end{proof}

We can now state the number of Euler liars for both cases in class B.

\begin{lemma}\label{lem:classB}
Let $n$ be a Carmichael number of class B.
\begin{enumerate}
\item If $n$ is in class B1 there are $\frac{1}{2^{h+1}}\varphi(n)$ Euler liars in $\mathbb{Z}_n^*$ and all
have Jacobi symbol 1.
\item If $n$ is in class B2 there are $\frac{1}{2^{k-1}}\varphi(n)$ Euler liars in $\mathbb{Z}_n^*$ and 1 and
-1 appear as Jacobi symbol equally often.
\end{enumerate}
\end{lemma}
\begin{proof}
(2) We start with class B2 as it is the easier case.
As all
$p_i-1$
have the same valuation at 2 as
$\lambda(n)$,
$a^{\frac{\lambda(n)}{2}}\pmod
n$ attains all square roots of 1 equally often,
in particular
$\pm
1$ are attained
$\frac{1}{2^k}\varphi(n)$
times each.
If
$a^{\frac{\lambda(n)}{2}}\equiv
1\pmod
n$ then by CRT
$a^{\frac{(p_i-1)}{2}}=1
\pmod
p_i$
for all
$1\le
i
\le
k$ and thus $a$ is a square modulo each
$p_i$
and by multiplicativity of the Jacobi symbol also
$\left(\frac{a}{n}\right)=1$.
Similarly,
if
$a^{\frac{\lambda(n)}{2}}\equiv
-1\pmod
n$ then $a$ is a non-square modulo each
$p_i$.
By
Lemma~\ref{lem:parity}
we have that $k$ is odd and thus
$\left(\frac{a}{n}\right)=(-1)^k=-1$.
This shows that there are
$2\frac{1}{2^k}\varphi(n)$
Euler liars in class B2 and both Jacobi symbols appear equally often.

(1) Let
$v_2(p_i-1)=v_2(\lambda(n))$
and let
$v_2(p_j-1)<v_2(\lambda(n))$.
Then
$\lambda(n)/2\equiv
0
\pmod{p_j-1}$
and
$a^{\frac{\lambda(n)}{2}}
\equiv
a^{\frac{(p_j-1)}{2}}\equiv
1\pmod{p_j}$
for all
$a\in
\Z_n^*$,
while both cases of
$a^{\frac{\lambda(n)}{2}}
\equiv
a^{\frac{(p_i-1)}{2}}\equiv
\pm
1\pmod{p_i}$
appear equally often.
Hence,
only the
$2^h$
square roots of 1 stemming from those
$p_i$
attaining the maximum valuation at 2 appear as
$a^{\frac{\lambda(n)}{2}}\pmod
n$ and in particular -1 does not appear by CRT.
Therefore,
Euler liars must match on $+1$.
Let
$S=\left\{a
\in
\Z_n^*\left|
a^{\frac{\lambda(n)}{2}}\equiv
1\pmod
n\right\}\right.$,
then
$|S|=\frac{1}{2^h}\varphi(n)$.
Let
$g_j$
be a generator of
$\Z_{p_j}^*$
and define
$a_j
\in
\Z_n^*$
via the CRT equations
$a_j\equiv
g_j\pmod{p_j}$
and
$a_j\equiv
1\pmod{p_\ell}$
for
$\ell
\ne
j$.
Then
$a_j\in
S$ by the definition of
$p_j$
and
$\left(\frac{a_j}{n}\right)=\left(\frac{g_j}{n}\right)=-1$
by multiplicativity of the Jacobi symbol and the definition of
$a_j$.
Let
$s\in
S$,
then
$a_js
\in
S$ and
$\left(\frac{a_js}{n}\right)=\left(\frac{a_j}{n}\right)\left(\frac{s}{n}\right)
=-\left(\frac{s}{n}\right)$,
showing that the Jacobi symbols in $S$ split evenly into
$\pm
1$.
This means that only half of the
$\frac{1}{2^h}\varphi(n)$
elements in $S$ are Euler liars.
\end{proof}

We
 summarize
the results in
Table~\ref{tab:car_eliars}.

\begin{table}
\begin{tabular}{|ccc|c|c|}
\hline
Class
&
parity of index $i(n)$
&
$h<k$?
&
\#
Euler liars
                    &
Jacobi symbols
         \\
\hline
A
   &
even
  &
N.A.
&
$
\frac{1}{2}
\varphi(n)
$
      &
all $ 1 $
              \\[1mm]
B1
  &
odd
  &
yes
  &
$
\frac{1}{2^{h+1}}
\varphi(n)
$
      &
all $ 1 $
              \\[1mm]
B2
  &
odd
  &
no
   &
$
\frac{1}{2^{k-1}}
\varphi(n)
$
&
half $ 1 $,
half $ -1 $
\\[1mm]
\hline
\end{tabular}
\caption{Number of Euler liars for the classes of Carmichael numbers. As in
Definition~\ref{def:classB},
$h$ is the number of factors
$p_i$
with
$v_2(\lambda(n))
=
v_2(p_i-1)$
for
$n=\prod_{i=1}^kp_i$.}
\label{tab:car_eliars}
\end{table}

\begin{remark}
Note that having one
$p_i\equiv
1
\pmod
4$ implies that
$n\equiv
1
\pmod
4$ as
$4|(p_i
-
1)|(n-1)$.
Hence,
$n\equiv
3
\pmod
4$ implies that
$p_i\equiv
3
\pmod 4$ for all $2\le i\le k$,
Table~\ref{tab:car_eliars}
shows that the two cases where $n$ being an Euler pseudoprime implies that it is a
strong pseudoprime,
namely for $n
\equiv
3
\pmod 4$ and for Euler liars matching on $-1$,  are both for class B2.
\end{remark}

In the next section we consider many examples of Carmichael numbers.
The first two rows of
Table~\ref{tab:types_stats}
there show that the split between class A and B is almost even,
with slightly more elements in class B when restricting to Carmichael numbers that
are not products of Carmichael numbers (there called ``atomic'').
When selecting Carmichael numbers that have a high chance of being Euler pseudoprimes
for many bases,
class A is obviously better than class B,
also B1 is never worse than B2.
The latter is obvious for small $h$,
but for $h = k-1$ the
$2^{h+1}$
in
Lemma~\ref{lem:classB}
would give a smaller proportion than for B2.
We now show that this case does not happen.

\begin{lemma}
Let
$n=\prod_{i=1}^k
p_i$
be a Carmichael number and let $h$ be the number of
$p_i$
with
$v_2(p_i
-1) =
v_2(\lambda(n))$.
Then
$h\in
\{1,2,3,\ldots
,
k-2,k\}$.
\end{lemma}

\begin{proof}
For a Carmichael number $n$ it holds that
$v_2(\lambda(n))
=
\max_{1\le
i
\le
k}\{v_2(p_i
-1)\}$,
hence,
$h
\ge
1$.
Let all
$p_i$
except for
$p_j$
attain this maximum and let $m =
v_2(p_j
- 1)$.
Writing $n$ as in the proof of
Lemma~\ref{lem:parity}
shows that $n
\equiv
1 +
2^m
\pmod{2^{m+1}}$,
contradicting that
$\lambda(n)|(n-1)$.
Hence,
there is an even number of
$p_i$
attaining the minimum,
except for the case $h=k$ when an odd number is possible;
the latter numbers are exactly those in class B2.
\end{proof}

Note that this lemma does not imply that $k-h$ is even.
The conditions on the 2-valuations are compatible with having two primes with $m$
and one with $m+1$ for
$v_2(n-1)
=
v_2(\lambda(n))
= m+2$,
giving $k-h = 3$.

In the following section we will consider products of Carmichael numbers.
We can use the classification to determine which class the product falls in if it
is a Carmichael number.

\begin{lemma}\label{le:classification-products}
Let
$n_1$
and
$n_2$
with
$\gcd(n_1,n_2)=1$
be Carmichael numbers.
The following are necessary conditions and impossibility results for $n =
n_1
n_2$
to be a Carmichael number:
\begin{itemize}
\item If $n_1$ and $n_2$ are both in class A and $n$ is a Carmichael number, then
\begin{itemize}
\item $n$ is in class A if $\min\{v_2(n_1 - 1), v_2(n_2-1)\} > \max\{v_2(\lambda(n_1)), v_2(\lambda(n_2))\}$;
\item $n$ is in class B1 if  $ \min\{v_2(n_1 - 1), v_2(n_2-1)\} = \max\{v_2(\lambda(n_1)), v_2(\lambda(n_2))\}$ and $v_2(n_1 - 1) \neq v_2(n_2-1)$.
\end{itemize}
Otherwise $n$ is not a Carmichael number.
\item If $n_1$ is in class A, $n_2$ is in class B, and $n$ is a Carmichael number, then
\begin{itemize}
\item $n$ is in class A if $v_2(n_1 - 1) = v_2(n_2-1)$;
\item $n$ is in class B1 if $v_2(n_1 - 1) > v_2(n_2-1)$ and
$v_2(\lambda(n_1))<
v_2(\lambda(n_2))$;
\item $n$ is in class B1 or B2 if $v_2(n_1 - 1) > v_2(n_2-1)$ and
$v_2(\lambda(n_1))=
v_2(\lambda(n_2))$.
\end{itemize}
Otherwise $n$ is not a Carmichael number.
\item If $n_1$ and $n_2$ are both in class B and $n$ is a Carmichael number, then
\begin{itemize}
\item $n$ is in class A if $v_2(n_1 - 1) = v_2(n_2-1)$.
\end{itemize}
Otherwise $n$ is not a Carmichael number.
\end{itemize}
\end{lemma}
\begin{proof}
The proofs for the different cases use that
$v_2(n
-1) =
v_2(n_1n_2
- 1) =
v_2(n_2(n_1
-1) +
n_2
- 1)
\geq
\min
\{v_2(n_1
- 1),
v_2(n_2-1)\}$.
Note that it is larger than the minimum if
$v_2(n_1
- 1) =
v_2(n_2-1)$.
If
$n=n_1
n_2$
is a Carmichael number then
$v_2(\lambda(n))
=
 \max\{v_2(\lambda(n_1)),
v_2(\lambda(n_2))\}$.

{\bf
Let
$n_1$
and
$n_2$
be in class
A.}
By definition,
their indices are even,
i.e.,
$v_2(n_1
-1)
>
v_2(\lambda(n_1))$
and
$v_2(n_2-1)
>
v_2(\lambda(n_2))$.
Let
$v_2(\lambda(n_1))
\geq
v_2(\lambda(n_2))$.

If
$v_2(n_1
-
1)\le
v_2(n_2-1)$
we have that
$v_2(n-1)\ge
\min
\{v_2(n_1
- 1),
v_2(n_2-1)\}
=
v_2(n_1
- 1)
>
v_2(\lambda(n_1))
=
 \max\{v_2(\lambda(n_1)),
v_2(\lambda(n_2))\}=v_2(\lambda(n))$.
Hence $n$ is in class A.
This case is included in
$\min\{v_2(n_1
- 1),
v_2(n_2-1)\}
>
\max\{v_2(\lambda(n_1)),
v_2(\lambda(n_2))\}$.

If
$v_2(n_1
- 1)
>
v_2(n_2-1)$
we have that
$v_2(n-1)
=
\min
\{v_2(n_1
- 1),
v_2(n_2-1)\}
=
v_2(n_2-1)
$.
This is larger than
$v_2(\lambda(n))$,
i.e.
$n$ is in class A,
iff
$\min\{v_2(n_1
- 1),
v_2(n_2-1)\}
>
\max\{v_2(\lambda(n_1)),
v_2(\lambda(n_2))\}$
and equals
$v_2(\lambda(n))$,
i.e.,
$n$ is in class B,
iff
$\min\{v_2(n_1
- 1),
v_2(n_2-1)\}
=
\max\{v_2(\lambda(n_1)),
v_2(\lambda(n_2))\}$.
In the latter case
$v_2(n_2-1)
>
v_2(\lambda(n_2))$
implies
$v_2(\lambda(n_1))
>
v_2(\lambda(n_2))$
and thus $n$ is in class B1.

Otherwise
$v_2(n-1)
<
v_2(\lambda(n))$
and $n$ is not a Carmichael number.

{\bf
 Let
$n_1$
be in class A and
$n_2$
be in class
B.}
By definition,
$v_2(n_1
-1)
>
v_2(\lambda(n_1))$
and
$v_2(n_2-1)
=
v_2(\lambda(n_2))$.

If
$v_2(n_1
- 1) =
v_2(n_2-1)$
we have that
$v_2(n-1)>v_2(n_2-1)
=v_2(\lambda(n_2))
=
v_2(\lambda(n))$,
because
$v_2(n_2
- 1) =
v_2(n_1-1)
 >
v_2(\lambda(n_1))$
implies that
$v_2(\lambda(n_1))
<
v_2(\lambda(n_2))
 =
v_2(\lambda(n))$,
and $n$ is in class A.

If
$v_2(n_1
- 1)
>
v_2(n_2-1)$
we have that
$v_2(n-1)
=
v_2(n_2-1)
=v_2(\lambda(n_2))\le
\max\{v_2(\lambda(n_1)),
v_2(\lambda(n_2))\}
=
v_2(\lambda(n))$.
If $n$ is a Carmichael number we must have equality,
i.e.,
$n$ is in class B,
which requires
$v_2(\lambda(n_1))
\le
v_2(\lambda(n_2))$.
If
$v_2(\lambda(n_1))
<
v_2(\lambda(n_2))$
there are at least two prime factors
$p_1|n_1$
and
$p_2|n_2$
with
$v_2(p_1-1)
<
v_2(p_2-1)$
and thus $n$ is in class B1.
If
$v_2(\lambda(n_1))
=
v_2(\lambda(n_2))$
the result is in B2 if all prime factors
$p_i$
of
$n_1$
and
$n_2$
have the same
$v_2(p_i
-1)$ and otherwise in B1.

If
$v_2(n_1
- 1)
<
v_2(n_2-1)$
we have that
$v_2(n-1)=v_2(n_1-1)
<
v_2(n_2-1)
=
v_2(\lambda(n_2))
\le \max\{v_2(\lambda(n_1)), v_2(\lambda(n_2))\} = v_2(\lambda(n))$, contradicting that
$n$ is a Carmichael number.

{\bf
 Let
$n_1$
and
$n_2$
be in class
B.}
By definition,
$v_2(n_1-1)
=
v_2(\lambda(n_1))$
and
$v_2(n_2-1)
=
v_2(\lambda(n_2))$.

If
$v_2(n_1
- 1) =
v_2(n_2-1)$,
we have
$v_2(n-1)
>
v_2(n_1-1)
=
v_2(\lambda(n_1))
=
v_2(\lambda(n))$,
so $n$ is in class A.
In all other cases we have
$v_2(n
-1)
<
v_2(\lambda(n))$,
contradicting that $n$ is a Carmichael number.
\end{proof}

\section{Construction of Euler Pseudoprimes}\label{sec:construct_carmichael}
In this section we will explain why we look towards Carmichael numbers to find Euler
pseudoprimes.
Then we will show how we created more and stronger Euler pseudoprimes by multiplying
them and give arguments for the choices we made in creating these.
To substantiate these arguments we will show some statistics on the found Euler pseudoprimes.
Furthermore we will share some interesting observations that we have made on the Euler
pseudoprimes we find in this way.

\subsection{Focus on Carmichael numbers}
By
Lemma~\ref{thm:
eulerfermat}
any Euler liar is also a Fermat liar,
making Carmichael numbers obvious choices for good Euler pseudoprimes,
beside the RSA application.
To test this approach,
we started the search for Euler pseudoprimes by looking at all composite numbers less
than some bound.
We quickly noticed that any composite number that survived for more than a couple
of bases was a Carmichael number.
Specifically when looking at all composite numbers up to
$10^{10}$,
any number that survives till $17$ or higher is a Carmichael number.

This is corroborated by our analysis on Carmichael numbers in
Section~\ref{ch:carmichael_numbers},
where we showed that Carmichael numbers in class A have exactly
$\frac{1}{2}\varphi(n)$
Euler liars.
By
Lemma~\ref{lem:max-phi}
this is the maximum number of Euler liars a composite number can have.

We will consider products of Carmichael numbers that are Carmichael numbers themselves.
In order to have a well-defined notion of how many Carmichael numbers are involved
in such a product we need a name for the simplest building block.

\begin{definition}[Atomic Carmichael number]
A Carmichael number that is not a product of other Carmichael numbers is called an
{\em
atomic Carmichael
number}.
\end{definition}

We use Shallue and Webster's list of Carmichael numbers up to
$10^{24}$~\cite{2025/webster}
as our input Carmichael numbers to test if they are good Euler pseudoprimes.
We first filtered them to select only the atomic Carmichael numbers.
Next,
since we are looking for Euler pseudoprimes that survive as many small prime bases
as possible,
we are not interested in Carmichael numbers that have a small factor,
since
(\ref{eq:ProofSSAssumption})
fails when $a$ is a prime factor of $n$.
Thus we filter our list of Carmichael numbers by removing all numbers whose smallest
factor is less than 150.
Note that some of the statistics we show are still for all atomic Carmichael numbers,
but most of the results will focus on the ones that do not have small factors.

We introduce the following notation for Euler pseudoprimes that pass Euler's criterion
exactly up to prime base $a$.
\begin{definition}[$\epsp(a)$]
The set of Euler pseudoprimes that pass all prime bases from 2 up to and including
$a$,
where $a$ is prime,
but not the next prime after $a$ is denoted by
$\epsp(a)$.
\label{def: epsp(a)}
\end{definition}

Using the filtered down list of Carmichael numbers we create sets of Euler pseudoprimes
that survive exactly up to a certain base $a$ and not beyond,
i.e.,
sets
$\epsp(a)$
for a prime $a$.
We only save the ones for $a
\geq
37$,
which are the ones that survive at least the first 12 primes.
We have found $129092$ Euler pseudoprimes that survived at least until $a = 37$,
and that have prime factors
$>150$.
For full numbers see the second column of
Table~\ref{tab:epsp_numbers}.
The best ones we have found are $$649155636985808498298481
\text{
and
}
663228827208624759722641$$ which both survive up to base $113$,
the 30-th prime base.

\begin{table}[]
\centering
\begin{tabular}{R{0.7cm}|R{1cm}|R{1.5cm}|R{1.75cm}|R{1.7cm}}
$a$
&
$\epsp(a)$
&
$\epsp2(a)$
&
$\epsp3_{[1,2]}(a)$
&
 $\epsp4_{[2]^2}(a)$
\\
\hline
37
&
61\,079
 &
44
&
0
&
0\\
41
&
31\,405
 &
185\,460
&
68
&
633
\\
43
&
16\,867
 &
143\,171
&
1\,525\,406
&
18\,381\,423\\
47
&
9\,102
 &
84\,244&
1\,351\,736&
20\,850\,077
\\
53
&
4\,628
 &
45\,926
&
851\,904&
14\,475\,599
\\
59
&
2\,502
 &
24\,074
&
473\,833
&
8\,441\,754\\
61
&
1\,176
 &
12\,357
&
 249\,287&
4\,548\,214
\\
67
&
668
&
6\,197
&
124\,626
&
2\,275\,521
\\
71
&
320
 &
3\,102
&
63\,096&
1\,156\,495
\\
73
&
157
&
1\,608
&
31\,747
&
583\,013
\\
79
&
100
&
820
&
15\,875
&
293\,754
\\
83
&
43
&
406
&
7\,816&
147\,010
\\
89
&
13
&
189
&
3\,956
&
73\,612
\\
97
&
9
&
102
&
2\,022
&
37\,783\\
101
&
5
&
49
&
1\,002
&
18\,343\\
103
&
3
&
33
&
509
&
9\,192
\\
107
&
1
&
18
&
239
&
4\,529
\\
109
&
0
&
3
&
136
&
2\,357
\\
113
&
2
&
3
&
60
&
1\,139
\\
127
&
0
&
3
&
40
&
568\\
131
&
0
&
0
&
22
&
292\\
137
&
0
&
0
&
5
&
151
\\
139
&
0
&
1
&
3
&
82
\\
149
&
0
&
0&
4
&
38
\\
151
&
0
&
0&
1&
17
\\
157
&
0
&
0&
2&
4
\\
163
&
0
&
0
&
0
&
3
\\
167
&
0
&
0
&
0
&
2
\\
173
&
0
&
0
&
0
&
2
\\
179
&
0
&
0
&
0
&
1
\\
191
&
0
&
0
&
0
&
1
\\
\hline
Total
&
128\,080
 &
507\,810
&
4\,703\,395
&
71\,300\,609
\\
\end{tabular}
\caption{Number of found Euler pseudoprimes (with smallest factor larger than 150) for base larger than 37.
For notation see
Definitions~\ref{def:
epspN}
and
\ref{def:epsp-complex}.
}
\label{tab:epsp_numbers}
\end{table}

Experimentally we have seen that certain constructions of Carmichael numbers,
like
Arnault\cite{Arnault1995}
or
Chernick~\cite{Chernick},
which only construct a subset of the Carmichael numbers,
do not seem to construct interesting Euler pseudoprimes.

Now that we have elements of
$\epsp(a)$,
we will show how we use these to find better ones by multiplying them.

\subsection{Multiplying Euler pseudoprimes}
Multiplying Euler pseudoprimes to create larger ones was an intuitive choice due to
the Jacobi symbol being multiplicative.
In this section we will show that it indeed works out well.

Euler pseudoprimes that are not a product of other Euler pseudoprimes we call atomic
Euler pseudoprimes,
similar as with the atomic Carmichael numbers.

We will now define some notation specific to products of elements of
$\epsp(a)$
for some prime $a$.

\begin{definition}[$\epsp N(a)$]
The set of Euler pseudoprimes that are a product of $N$ atomic Euler pseudoprimes
is denoted by
$\epsp
N(a)$.
The elements of this set pass all prime bases up to and including $a$,
but not the next prime after $a$.
\label{def: epspN}
\end{definition}
Note that
$\epsp(a)$
and
$\epsp1(a)$
are the same and we will omit the 1.

Squaring an Euler pseudoprime does us no favors,
since as soon as the composite number contains a square it is not longer a Carmichael
number.
While the Jacobi symbol
$\left(\frac{a}{n^2}\right)$
is 1 for
$\gcd(a,n)=1$,
we have that
$\lambda(n^2)
=
n\lambda(n)$,
for $n$ a Carmichael number,
and
$n\not|
(n^2
-1)$,
hence
$a^{(n^2-1)/2}
\pmod{n^2}$
has a low chance of being 1.
So we will only focus on multiplying distinct Euler pseudoprimes.
\relax

\subsection{Link to Classification of Carmichael Numbers}\label{sec:link}
Table~\ref{tab:types_stats}
show statistics on the distribution of Carmichael numbers over the different classes
identified in
Section~\ref{ch:carmichael_numbers}.
When looking at the original set of atomic Carmichael numbers we see that the numbers
are pretty balanced between class A and B.
These statistics do not change when only looking at the atomic Carmichael numbers
with factors over $150$.
However,
among the found Euler pseudoprimes,
which are all also Carmichael numbers,
we find an overwhelming proportion of Carmichael numbers in class A and this proportion
increases when considering
$\epsp2(a)$.

\begin{table}
\centering
\begin{tabular}{m{3cm}|c|c|c}
Class
&
A
&
B1
&
 B2\\
\hline
Atomic Carmichael numbers
&153\,917\,210
 &
153\,669\,743&
489\,249\\
\hline
Atomic Carmichael numbers with smallest factor
$>150$
&7\,029\,697
&6\,978\,344
&
271\,231\\
        \hline
Found elements of
$\epsp(a)$
with $a
>
37$ and factor
$>150$
&
128\,063&
17
&
0\\
          \hline
Found elements of
$\epsp2(a)$
with $a
>
37$ and factor
$>150$
&
507\,767
&
43
&
0\\
\end{tabular}
\caption{Statistics on the Carmichael numbers and Euler pseudoprimes we
found per class.
There are
308\,279\,939
Carmichael numbers up to
$10^{24}$
of which
308\,076\,202
are
atomic.}
\label{tab:types_stats}
\end{table}

This makes sense since we saw in
Table~\ref{tab:car_eliars}
that class A has the highest number of Euler liars with
$\varphi(n)/2$
Euler liars,
so almost half of all possible bases are Euler liars (slightly less,
as
(\ref{eq:ProofSSAssumption})
fails when $a$ is a prime factor of $n$).
These numbers have chance
$1/2^i$
of passing
(\ref{eq:ProofSSAssumption})
for the $i$ first consecutive primes.
This matches
Table~\ref{tab:epsp_numbers},
where
$\epsp(a)$
for prime $a$ has approximately twice as many elements as
$\epsp(\text{NextPrime}(a))$.
By
Lemma~\ref{le:classification-products},
products that are Carmichael numbers are also likely to be in class A,
explaining the proportion in
Table~\ref{tab:types_stats}
and the exponential decline in sizes of
$\epsp2(a)$
in
Table~\ref{tab:epsp_numbers}.
In
Figure~\ref{fig:log_2_epsp}
we visualize this by plotting
$\log_2(|\epsp2(a)|)$
relative to how many prime bases it survived.

\begin{figure}[h]
\centering
\includegraphics[width=0.9\linewidth]{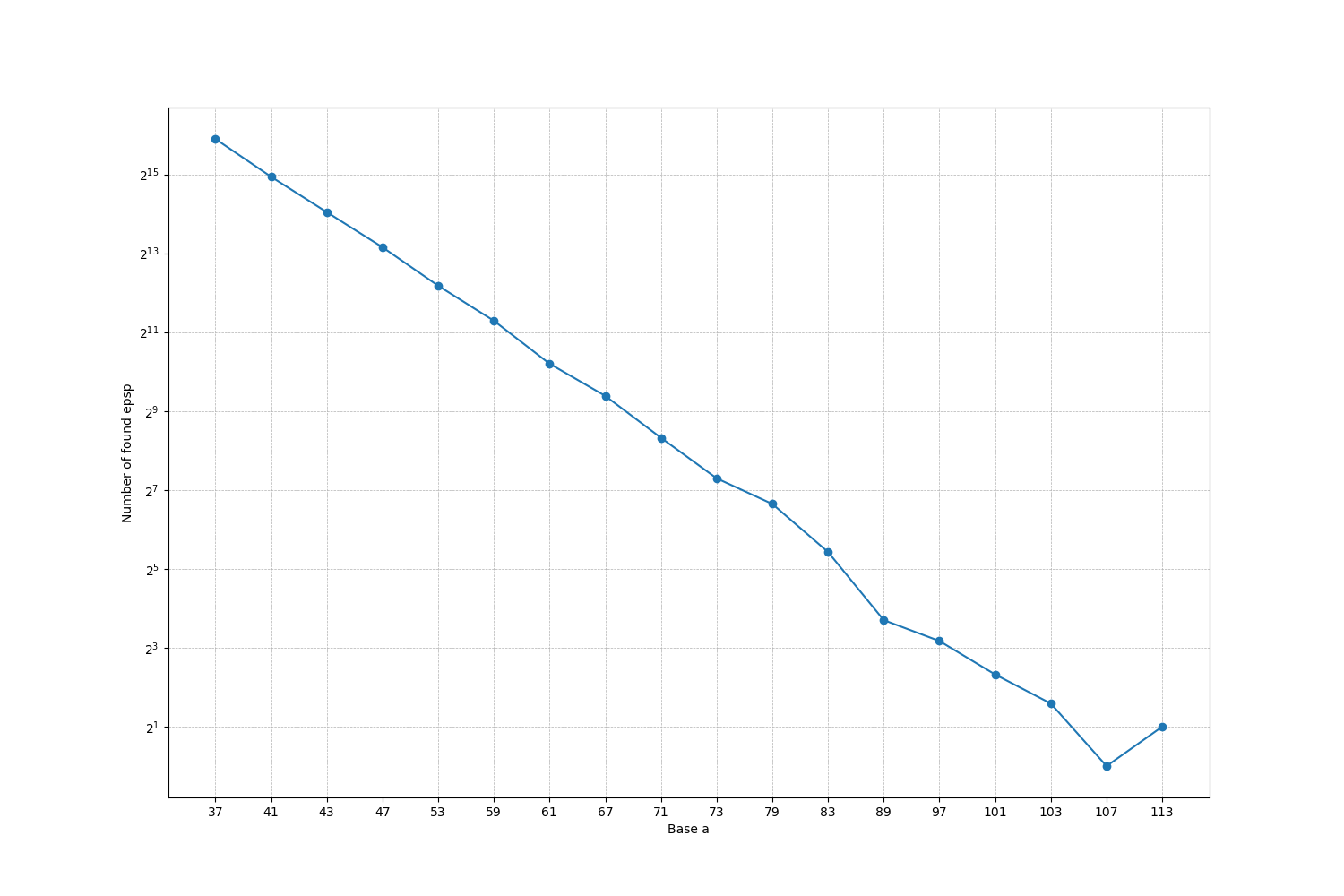}
\caption{Size of $\epsp(a)$ on logarithmic $y$-axis relative to $a$.}
\label{fig:log_2_epsp}
\end{figure}

Note that this trend is not always visible in our results,
e.g.
at the right end of the graph.
As another example,
consider the last column of
Table~\ref{tab:epsp_numbers},
which shows the sizes of
$\epsp4_{[2]^2}(a)$.
(This uses notation introduced in the next section to denote
$\epsp4(a)$
generated as products of two elements of
$\epsp2(a')$.)
Here
$\epsp4_{[2]^2}(47)$
has
{\bf
more}
elements than
$\epsp4_{[2]^2}(43)$.
This is because we are only multiplying
$\epsp2(a)$
for
$a\geq
37$,
so we are missing a lot of elements of
 $\epsp4_{[2]^2}(43)$
which would be created by multiplying
$\epsp4_{[2,2]}(a)$
for $a
<
37$.

\subsection{Choice of Base for Multiplying}\label{sec:choice_basis}
In this section we will explain why we will only multiply Euler pseudoprimes that
survive up to exactly the same base.
We use observations here that we made by creating
$\epsp2(a)$
by simply multiplying numbers of
$\epsp(a)$
(for $a
\geq
37)$ and checking if
(\ref{eq:ProofSSAssumption})
holds.

Experimentally we noticed quickly that multiplying two elements from the same set
$\epsp(a)$
for some $a$ gives us better results than multiplying between
$\epsp(a_1)$
and
$\epsp(a_2)$
with
$a_1
\neq
a_2$.
In
Table~\ref{tab:epsp_37}
we can see an example of what made us decide to focus on multiplying numbers that
survive until the same base.
We can see that when multiplying elements of
$\epsp(37)$
with each other we find better Euler pseudoprimes,
even up to prime base $139$.
Also note that we do not find many elements of
$\epsp2(37)$,
which will be explained later.

\begin{table}[]
\centering
\begin{tabular}{c|c|c}
&
From
$\epsp(37)$
&
from
$\epsp(37)
\cdot
\epsp(a),
a
\geq
41$
\\
\hline
$\#$
potential new numbers
&
$1\,895\,586\,378$&
$4\,157\,347\,387$\\
\hline
$\epsp2(37)$
&
44
&
760\,960\\
$\epsp2(41)$
&
188\,732
&
5
\\
$\epsp2(43)$
&
95\,919
&3
\\
$\epsp2(47)$
&47\,939
&
1\\
$\epsp2(53)$
&
24\,071
&2\\
$\epsp2(59)$
&12\,202
&1\\
$\epsp2(61)$
&
6\,279
&0\\
$\epsp2(67)$
&
3\,059&0\\
$\epsp2(71)$
&
1\,564&0\\
$\epsp2(73)$
&
813
&0\\
$\epsp2(79)$
&
407&0\\
$\epsp2(83)$
&
201
&0\\
$\epsp2(89)$
&
81
&0\\
$\epsp2(97)$
&
46
&0\\
$\epsp2(101)$
&
18
&0\\
$\epsp2(103)$
&
18
&0\\
$\epsp2(107)$
&
6
&0\\
$\epsp2(109)$
&
0
&0\\
$\epsp2(113)$
&
0
&0\\
$\epsp2(127)$
&
2
&0\\
$\epsp2(139)$
&
1
&0\\
\hline
Total
&381\,402
&
760\,972
\\
$\%$
again Euler pseudoprime
&
$0.020\%$
&
$0.018
\%$\\
\end{tabular}
\caption{Comparing creating $\epsp2(a)$ from either same bases or different bases from $\epsp(37)$, with smallest factor larger than 150.}
\label{tab:epsp_37}
\end{table}

Meanwhile when we look at the multiplication of all elements of
$\epsp(37)$
with all other
$\epsp(a)$
for $a
\geq
41$ we find in essence only new Euler pseudoprimes that survive up to $37$ and hardly
any that go beyond.
Multiplying elements from different bases does result in an absolute higher number
of Euler pseudoprimes (though the fraction of surviving is slightly lower),
but they tend to hardly ever survive past the original base.
If one needs a lot of elements that survive to the same base though this might be
useful.
We explain the first part in the following lemma.

\begin{lemma}\label{le:one-more}
Let
$n_1\in
\epsp
N_1(a)$
and
$n_2
\in
\epsp
N_2(a)$
both be Carmichael numbers in class $A$.
Let $b$ the next prime after $a$ and let
$\gcd(b,n_1)
=\gcd(b,n_2)=1$.
Then if the product $n =
n_1n_2$
is a Carmichael number in class A,
it is an Euler pseudoprime for all primes up to and including at least $b$.
\end{lemma}

\begin{proof}
Carmichael numbers in class $A$ have even index,
hence,
the exponentiation side of
(\ref{eq:ProofSSAssumption})
is always $1$ for
$a\in
\Z_n^*$.
Because
$n_i
\in
\epsp
N_i(a)$
we know that for all primes
$a'\le
a$ also
$\left(\frac{a'}{n_i}\right)
= 1$.
Because
 (\ref{eq:ProofSSAssumption})
fails for $b$ and the
$n_i$
are coprime to $b$,
we have
$\left(\frac{b}{n_i}\right)
= -1$.

If $n =
n_1n_2$
is a Carmichael number in class $A$ then
$a^{\frac{n-1}{2}}
\equiv
1
\mod
n$ for all
$a\in
\Z_n^*$.
For $a'
\leq
a$ we have
$\left(\frac{a'}{n}\right)
=
\left(\frac{a'}{n_1}\right)
\left(\frac{a'}{n_2}\right)
= 1
\cdot
1 =
 1$,
so $n$ passes
 (\ref{eq:ProofSSAssumption})
for these bases,
and
$\left(\frac{b}{n}\right)
=
\left(\frac{b}{n_1}\right)
\left(\frac{b}{n_2}\right)
= -1
\cdot
(-1) =
 1$.
Thus $b$ is an Euler liar.
\end{proof}

Note that the product of two Carmichael numbers from class A is not always a Carmichael
number,
but by
 Lemma~\ref{le:classification-products}
the overwhelming majority of the products that are Carmichael numbers will be in class
A.
By
Table~\ref{tab:types_stats}
most factors,
i.e.,
elements of
$\epsp
N(a)$,
are in class A,
as the rest usually failed at a smaller base already.
This also explains why we only found only a few elements of
$\epsp2(37)$
(see
Table~\ref{tab:epsp_numbers}),
since we do not multiply elements of
$\epsp(a)$
for $a
<
37$.

On the other hand,
when multiplying
$n_1
\in
\epsp
N_1(a_1)$
with
$n_2
\in
\epsp N_2(a_2)$, with $a_1 < a_2$, if the product $n_1n_2$ is a Carmichael number
in class A then it will continue to survive up to
$a_1$.
But for the next base $a'$ (possibly equal to
$a_2$),
we will always have that
$\left(\frac{a'}{n_1}\right)
= -1$ and
$\left(\frac{a'}{n_2}\right)=1$,
thus
$\left(\frac{a'}{n_1n_2}\right)
= -1$.
Hence,
(\ref{eq:ProofSSAssumption})
does not hold for
$n_1n_2$
at $a'$,
and thus
$n_1n_2$
will be an element of
$\epsp2(a_1)$,
just like
$n_1$.
This explains the large
 number
of
$\epsp2(37)$
in the last column of
Table~\ref{tab:epsp_37}.

\subsection{Choice of Class for Factors}
Figure~\ref{fig:distribution_epsp2}
shows
$\epsp2(a)$
that are products of
$\epsp(37)$,
(except the 1864371897 products that do not even pass base $2$).
Note that for all $a$ up to 139,
$\epsp2(a)$
has at least one element in them.
There is a clear distinction though between two large groups.

\begin{figure}
\centering
\includegraphics[width=1.2\linewidth]{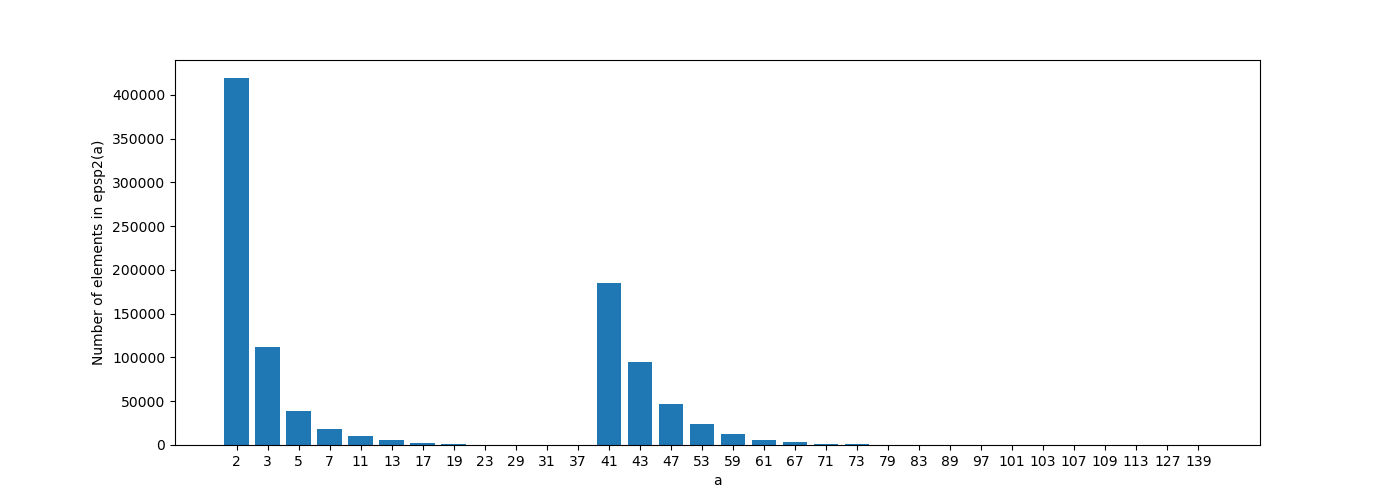}
\caption{Number of elements of $\epsp2(a)$ after multiplying two elements of $\epsp(37)$}
\label{fig:distribution_epsp2}
\end{figure}

The first hump at 2 and its tail are numbers that are either not Carmichael numbers,
but pass some small bases,
or unlucky ones that do not even survive the small bases.
By
Lemma~\ref{le:one-more}
the latter does not happen if both inputs are in class A,
as those survive at least till the next prime,
here 41,
matching the next hump and tail starting at 41.
Numbers on this side are likely in class A

An interesting note is that when one looks at all Carmichael numbers,
i.e.,
not filter out the ones with small factors,
one does find more elements of
$\epsp2(a)$
when multiplying two elements from
$\epsp(a)$
for the same $a$.
These are mostly the numbers that have the next prime larger than $a$ as a factor.
A very small minority of these are in class A and some in class B1.

We note that the proof of
Lemma~\ref{le:one-more}
also works for showing that products of Carmichael numbers involving one or two factors
from class B1 and leading to a Carmichael number in class A also reach the next base
after $a$,
unless that prime is a factor of one of the numbers.

Any product involving an element of B2 has only chance
$1/2^i$
of reaching the $i$-th prime base,
even if both factors individually reached it and the product is a Carmichael number
in class A.
This is because they reached this prime by matching on $1$ or $-1$ and thus need to
have matching signs from the other factor,
because in class A the exponentiation side always equals 1.
However,
we do not observe this because by
Table~\ref{tab:types_stats}
there is no element of
$\epsp(37)$
that is in class B2.

Products landing in class B1 are not guaranteed to reach the same base because the
exponentiation side is not guaranteed to equal 1.
If both factors are in class A with
$v_2(\lambda(n_1))>v_2(\lambda(n_2))$
then
$(n_1n_2-1)/\lambda(n_2)$
is even.
Following the same argumentation as in the proof of
Lemma~\ref{lem:classB},
but with primes
$p_i$
and
$p_j$
replaced by
$n_1$
and
$n_2$
we have that the exponentiation side always holds modulo
$n_2$
but takes on different square roots of 1 modulo
$n_1$.
For primes up to and including the next prime after $a$,
the Jacobi symbol for
$n_1n_2$
is 1,
ensuring that that side passes for the product.
Hence,
each prime base starting with 2 has probability
$1/2^{h}$
of passing,
where $h$ is upper bounded by the number of factors in
$n_1$.

Most atomic Carmichael numbers,
that are also good Euler pseudoprimes,
are in class A.
By
Lemma~\ref{le:classification-products}
a necessary condition for their product to be in class A is that $
\min\{v_2(n_1
- 1),
v_2(n_2-1)\}
>
\max\{v_2(\lambda(n_1)),
v_2(\lambda(n_2))\}$.
Restricting to inputs and outputs in class A misses only very few Carmichael numbers
that can contribute to good Euler pseudoprimes and speeds up selecting inputs with
good success chances.
We use these considerations as the basis for our algorithm to find good Euler pseudoprimes
in
Section~\ref{ch:results}.
\relax

\relax

\subsection{Other Options for Filtering Inputs}

Table~\ref{tab:mod120_stats}
shows the distribution of the residues modulo 120 of atomic Carmichael numbers,
for all of them and for those with smallest prime factor greater than 150,
and the found elements of
$\epsp(a)$
and
$\epsp2(a)$
with $a
>
37$ and smallest prime factor greater than 150.

\begin{table}
\centering
\begin{tabular}{m{3cm}|c|c|c|c|c|c|c}
$n
\pmod{120}$
&
01
&
25
&
49
&61
&
73
&
97
&
others
($<1\%$)\\
\hline
Percentage of atomic Carmichael numbers
 &
93.05
&
1.71
&
0.80
&
2.58
&
0.79
&
0.79
&
0.28\\
        \hline
Percentage of atomic Carmichael numbers with smallest factor
$>150$
&
88.30
&
0
&
1.87
&
4.52
&
1.66
&
1.71
&
1.94\\
        \hline
Percentage of found elements of
$\epsp(a)$
with $a
>
37$ and factor
$>150$
&
98.13
&
0
&
1.87
&
0
&
0
&
0
&
0\\
          \hline
Percentage of found elements of
$\epsp2(a)$
with $a
>
37$ and factor
$>150$
&
99.97
&
0
&
0.03
&
0
&
0
&
0
&
0\\
\end{tabular}
\caption{Statistics on the numbers (mod 120) for atomic Carmichael numbers and found $\epsp(a)$ and $\epsp2(a)$ with $a > 37$ and factor $>150$.}
\label{tab:mod120_stats}
\end{table}

The vast majority of all numbers studied is congruent to 1 modulo 120.
For atomic Carmichael numbers this holds
93.05\%
of the time,
with the second highest number,
but just
2.58\%,
for residue 61 (reaching 1 modulo 4 but not modulo 8),
and the third with
1.71\%
for residue 25.
Unsurprisingly the latter disappears when filtering out numbers with small factors,
as all of these are divisible by 5.
While the filtered numbers still have a peak at residue 1,
is is slightly less pronounciated,
likely due to small prime factors that contributed more to this than to the others.

For the Euler pseudoprimes (elements of
$\epsp(a)$
or
$\epsp2(a)$),
only residues 1 and 49 modulo 120 remain,
with 1 appearing
98.13\%
of the time for elements of
$\epsp(a)$
and
99.97\%
of the time for elements of
$\epsp2(a)$.

This is easy to explain because all of these are in class A or B1,
i.e.,
pass
(\ref{eq:ProofSSAssumption})
for Jacobi symbol 1.
This requires that
$n\equiv
\pm
1
\pmod
8,
n
\equiv
1
\pmod
3$,
and $n
\equiv
\pm
1
\pmod
5$.
This rules out 61,
73,
and 97.
It is no surprise that the remaining residues,
1 and 49,
are 1 modulo 8,
as 7 modulo 8 would imply
$n\equiv
3
\pmod
4$ with $n$ landing in class B2,
none of which made it till prime 37.
Hence,
we can filter any new Carmichael numbers by their residue class modulo 120.

Carmichael numbers congruent to 1 modulo many small primes have Jacobi symbol 1 at
those primes and thus satisfy
(\ref{eq:ProofSSAssumption})
if $n$ is in class A.
It may be tempting to prescribe
$n\equiv
1\pmod
L$ for $L$ some product of small primes,
however,
this may also seem to restrict the input numbers too much.
In
Table~\ref{tab:epsp-1
divisibility per
base}
we show statistics on divisors of $n-1$ for small primes and primepowers.
Note the law of small numbers for statistics for large $a$.

\begin{table}
\centering
\begin{tiny}
\begin{tabular}{r|r|r|r|r|r|r|r|r|r|r|r|r}
$a$
&
$2^4$
&
$2^5$
&
$3^2$
&
$3^3$
&
$5$
&
$5^2$
&
$7$
&
$7^2$
&
$11$
&
$13$
&
$17$
&
$19$
\\
\hline
37
 &
95.07
&
82.61
&
88.53
&
60.59
&
98.16
&
51.68
&
91.86
&
29.21
&
73.39
&
64.82
&
52.45
&
43.41
\\
41
 &
95.08
&
82.72
&
88.14
&
59.78
&
98.12
&
52.51
&
91.75
&
29.43
&
72.44
&
64.46
&
52.14
&
44.14
\\
43
 &
95.09
&
82.33
&
88.35
&
60.18
&
98.15
&
51.51
&
91.85
&
29.15
&
72.78
&
65.00
&
52.49
&
43.33
\\
47
 &
95.03
&
82.29
&
87.74
&
59.03
&
98.04
&
51.98
&
91.98
&
29.17
&
73.61
&
63.32
&
52.58
&
43.97
\\
53
 &
94.73
&
81.98
&
88.55
&
60.33
&
97.90
&
52.36
&
91.68
&
29.08
&
72.08
&
64.15
&
52.85
&
42.91
\\
59
 &
94.24
&
81.57
&
87.37
&
58.31
&
98.20
&
51.72
&
92.21
&
28.94
&
73.82
&
63.59
&
51.98
&
41.93
\\
61
 &
95.24
&
82.31
&
88.18
&
60.12
&
98.30
&
51.45
&
89.80
&
27.89
&
72.11
&
62.16
&
51.45
&
42.26
\\
67
 &
95.98
&
82.04
&
88.83
&
58.38
&
97.90
&
48.35
&
92.37
&
27.10
&
71.26
&
63.77
&
49.85
&
41.02
\\
71
 &
93.75
&
81.88
&
84.06
&
57.19
&
99.06
&
51.25
&
90.31
&
30.94
&
74.38
&
65.00
&
53.44
&
43.75
\\
73
 &
94.90
&
84.08
&
87.26
&
56.05
&
98.09
&
56.89
&
90.45
&
31.21
&
67.52
&
55.41
&
55.41
&
38.22
\\
79
 &
96.00
&
87.00
&
84.00
&
59.00
&
97.00
&
51.00
&
90.00
&
31.00
&
70.00
&
54.00
&
47.00
&
38.00
\\
83
 &
97.67
&
81.40
&
86.05
&
67.44
&
90.70
&
48.84
&
90.70
&
30.23
&
83.72
&
72.09
&
41.86
&
44.19
\\
89
 &
100.00
&
76.92
&
84.62
&
38.46
&
100.00
&
53.85
&
100.00
&
23.08
&
92.31
&
84.62
&
38.46
&
53.85
\\
97
 &
100.00
&
100.00
&
100.00
&
77.78
&
100.00
&
55.56
&
88.89
&
22.22
&
44.44
&
44.44
&
44.44
&
44.44
\\
101
&
100.00
&
100.00
&
80.00
&
60.00
&
100.00
&
40.00
&
100.00
&
0.00
&
80.00
&
60.00
&
60.00
&
20.00
\\
103
&
100.00
&
100.00
&
68.67
&
33.33
&
100.00
&
33.33
&
66.67
&
0.00
&
100.00
&
66.67
&
66.67
&
33.33
\\
107
&
0.00
&
0.00
&
100.00
&
100.00
&
100.00
&
0.00
&
100.00
&
0.00
&
100.00
&
100.00
&
100.00
&
0.00
\\
113
&
100.00
&
0.00
&
100.00
&
100.00
&
100.00
&
0.00
&
100.00
&
50.00
&
50.00
&
100.00
&
100.00
&
50.00
\\
\end{tabular}
\end{tiny}
\caption{Statistics on the divisibility of $n-1$ for $n\in \epsp(a)$ with prime factors $>150$ per base $a$.}
\label{tab:epsp-1 divisibility per base}
\end{table}

The statistics show that many Euler pseudoprimes are congruent to 1 for all or most
small primes.
In fact,
Erdös
in~\cite{ErdosCarmichaelConstruction}
proposed to construct Carmichael numbers that are 1 modulo $L$ for some smooth $L$
so that all
$p_i-1$
divide $L$,
guaranteeing that $n$ is a Carmichael number.
In
Table~\ref{tab:epsp-1
divisibility}
we collect statistics on $n$ modulo $L$ for some choices of $L$ showing that these
$n$ would have been constructed using
Erdös'
method and are good Euler pseudoprimes.
This shows an alternative route to generating large Euler pseudoprimes.

\begin{table}[h]
\centering
\begin{tabular}{l|r}
Statistic
&
Percentage
\\
\hline
\% of $n-1$ divisible by $2^4 \cdot 3^2 \cdot 5 \cdot 7$ &  75.67 \\
\% of $n-1$ divisible by $2^4 \cdot 3^2 \cdot 5 \cdot 7 \cdot 11 \cdot 13$ &35.02 \\
\% of $n-1$ divisible by $2^4 \cdot 3^2 \cdot 5 \cdot 7 \cdot 11 \cdot 17$ & 28.09 \\
\% of $n-1$ divisible by $2^4 \cdot 3^2 \cdot 5 \cdot 7 \cdot 11 \cdot 13 \cdot 17$ & 17.32 \\
\% of $n-1$ divisible by $2^6 \cdot 3^3 \cdot 5 \cdot 7 \cdot 11 \cdot 13 \cdot 17$ & 7.08 \\
\end{tabular}
\caption{Statistics on the divisibility of $n-1$ for $n\in\epsp(a)$ with prime factors $>150$ by certain $L$.}
\label{tab:epsp-1 divisibility}
\end{table}

\section{Finding large Euler pseudoprimes} \label{ch:results}
In this section we will showcase the good Euler pseudoprimes we have found in our
search.
We will start with the algorithm that we have created based on the analysis of
Sections~\ref{ch:carmichael_numbers}
and~\ref{sec:construct_carmichael}.

\subsection{Algorithm for finding Euler pseudoprimes} \label{sec:algorithm}
Just doing all multiplications,
even when getting rid of Carmichael numbers with small factors and for
$a\geq
37$,
very quickly takes too much computational power as the number of combinations grows
quadratically with each factor.
In this section we will discuss how for anything larger than
$\epsp2(a)$
we are still able to quickly find the Euler pseudoprimes that we are looking for.
First,
we restrict to the case identified above of multiplying two numbers in class A and
aiming for class A,
where both factors satisfy
(\ref{eq:ProofSSAssumption})
up to the same prime base $a$.

Second,
we organize our data to facilitate scanning and a series of tests of increasing complexity
to quickly discard products.
We create one file per $(N,a)$ for
$\epsp
N(a)$ and within each file order our numbers by
$v_2(\lambda(n))$.
This means that when doing multiplication between elements
$n_1$
and
$n_2$
we easily ensure that
$\max\{v_2(\lambda(n_1)),
v_2(\lambda(n_2))\}
=
v_2(\lambda(n_1))$.
By
Lemma~\ref{le:classification-products}
we need
$\min\{v_2(n_1
- 1),
v_2(n_2-1)\}
>
\max\{v_2(\lambda(n_1)),
v_2(\lambda(n_2))\}$
to land in class A.
By
$n_1$
being in class $A$ we already have that
$v_2(n_1
- 1)
>
v_2(\lambda(n_1))$.
Because of this we only have to check if
$v_2(n_2-1)
>
v_2(\lambda(n_1))$
to know we are in the correct case.
To make this cheap,
we keep track of
$v_2(n-1),
v_2(\lambda(n))$
of our found Euler pseudoprimes.

After doing this cheap check we compute
$\gcd(n_1,n_2)$
 to
eliminate cases where
$n_1$
and
$n_2$
share a prime factor.
For the remaining numbers we use
Lemma~\ref{lem:alter_carm}
to check that they are Carmichael number by testing if
$\lambda(n)|
(n-1)$.
We speed this up by computing
$\lambda(n)$
as
$\lambda(n)
=
\lcm(\lambda(n_1),\lambda(n_2))$.
We can do this efficiently by also keeping track of the
$\lambda(n)$
for all Euler pseudoprimes we find.
\relax

Now that we are for certain in the case that $n$ is a Carmichael number in class A,
the only thing we have to do is to check till what base it survives.
We know for certain that it survives up to and including
 base
$a$,
but by
Lemma~\ref{le:one-more}
it also for certain survives the next base as well.
So we do not have to start checking until the one after that.
Because we know $n$ has even index we only have to check the Jacobi symbol
 $\left(\frac{a}{n}\right)$
starting at the second prime after $a$.
This also speeds up the process,
since computing the Jacobi symbol is quite cheap.

The exact algorithm we used can be found in
Algorithm~\ref{alg:find_epsp}.

\begin{algorithm}
\caption{\textbf{Finding Class A Euler Pseudoprimes}}
\begin{flushleft}
\textbf{Input:} Prime $a$, $n_1 \in \epsp N_1(a), n_2\in \epsp N_2(a)$  both in class A
with
$v_2(\lambda(n_1))
\geq
v_2(\lambda(n_2))$,
with known
$v_2(n_i-1),
v_2(\lambda(n_i)),
\lambda(n_i)$.
\relax

\textbf{Output:} $a',n, \lambda(n),v_2(n-1),v_2(\lambda(n))$ such that $n = n_1n_2 \in \epsp N(a')$ with $N=N_1+ N_2$ or None if $n$ is not a class A Carmichael number.
\end{flushleft}
\hrulefill
\begin{algorithmic}

\If{$v_2(n_2 - 1) > v_2(\lambda(n_1))$}
\If{$\gcd(n_1,n_2) = 1$}
\relax
\State $\lambda(n) \gets \lcm(\lambda(n_1),\lambda(n_2))$
\State $n \gets n_1 \cdot n_2$
\If{$\lambda(n)|(n-1)$}
\State $a \gets \text{NextPrime}(a)$
\State $b \gets \text{NextPrime}(a)$
\While{$\left(\frac{b}{n}\right)$ = 1}
\State $a \gets b$
\State $b \gets \text{NextPrime}(a)$
\EndWhile
\State \textbf{return} $a,n, \lambda(n),v_2(n-1),v_2(\lambda(n))$
\EndIf
\EndIf
\EndIf
\State \Return None
\end{algorithmic}
\label{alg:find_epsp}
\end{algorithm}

In practice,
we replace the function NextPrime by storing the primes sorted by size in an array.
Instead of passing $a$,
its position is input and the NextPrime calls just increments the index.

\subsection{Search for more Euler pseudoprimes}
In this section we will go over the Euler pseudoprimes that we have found using
Algorithm~\ref{alg:find_epsp}
and the choices we made for the
$N_i$
and structure of building larger numbers.

So far we have reported on products of two atomic Euler pseudoprimes,
both in the same
$\epsp(a_1)$,
giving an element of
$\epsp2(a)$,
which we depict in
Figure~\ref{fig:epsp2}.
\begin{figure}[ht]
\centering
\includegraphics[width=0.4\linewidth]{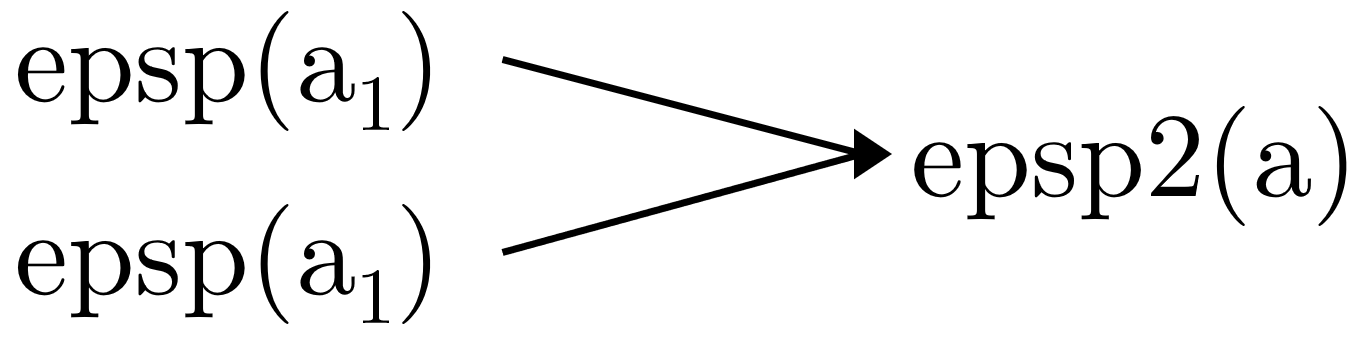}
\caption{Schematic way of how we create $\epsp2(a)$}
\label{fig:epsp2}
\end{figure}

The best Euler pseudoprime found from considering all
$a_1
\ge
37$ is
$${\tiny
64246463686535715315961618118197305491452801}$$
which survives till 139,
the 34th prime base.

Now that we have organized the elements
 of
$\epsp2(a)$
and
$\epsp(a)$
there are two ways to continue.
We start by creating
$\epsp3(a)$
by multiplying elements of
$\epsp(a_2)$
with elements of
$\epsp2(a_2)$,
see
Figure~\ref{fig:epsp3}.
The size of each resulting
$\epsp3(a)$
 can
be found in the fourth column of
Table~\ref{tab:epsp_numbers}.

\begin{figure}[h]
\centering
\includegraphics[width=0.6\linewidth]{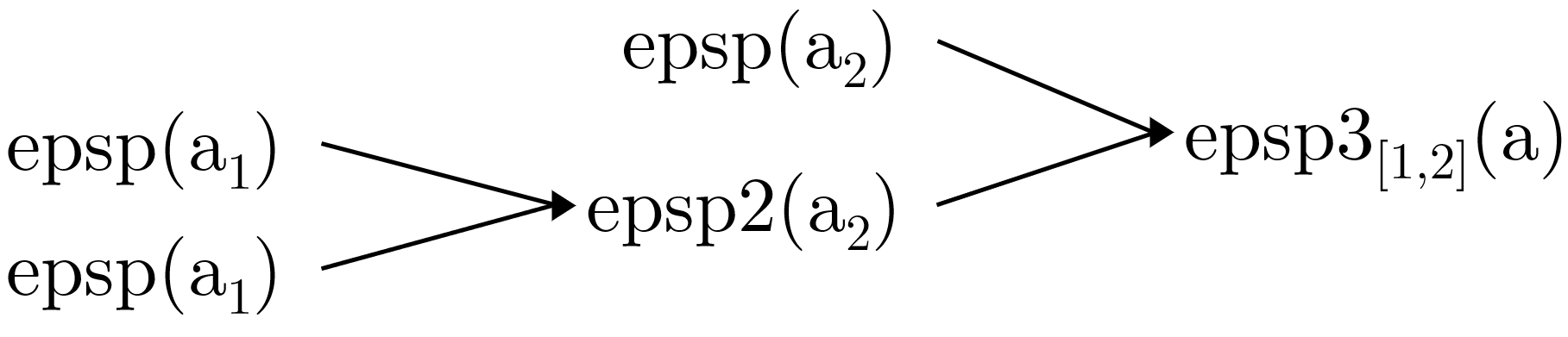}
\caption{Schematic way of how we create $\epsp3_{[1,2]}(a)$}
\label{fig:epsp3}
\end{figure}

In more generality,
for each
$\epsp
N_1(a_1)$
 we
can use
Algorithm~\ref{alg:find_epsp}
on all
$n_1,n_2\in
\epsp
N_1(a_1)$,
where
$n_2$
comes later in the file than
$n_1$.
If the algorithm output is not None,
we append $n,
\lambda(n),v_2(n-1),v_2(\lambda(n))$
to the file for
$\epsp2N_1(a)$.
Once all computations are done,
we sort each resulting file by
$v_2(\lambda(n))$.
In full generality,
we do the above with all pairs in
$\epsp
N_1(a_1)
\times
\epsp
N_2(a_1)$
(note the same
$a_1$),
where we relabel the inputs if necessary such that
$v_2(\lambda(n_1))\ge
v_2(\lambda(n_2))$.
But we note that this notation loses information,
e.g.,
to compute numbers that are products of four atomic Euler pseudoprimes we could multiply
two elements of
$\epsp2(a)$
or one element of
$\epsp(a)$
with one of
$\epsp3(a)$.

We introduce the following notation to indicate more precisely than in
Definition~\ref{def:
epspN}
how a new Euler pseudoprime is constructed from the
$\epsp
N(a)$.

\begin{definition}[$\epsp N_{[T_1,T_2]}(a)$]\label{def:epsp-complex}
The notation
$\epsp
N_{[T_1,T_2]}(a)$
is defined
 recursively.
Each
$T_i$
is of the form
$T_i
=
M_
i
[S_{i1},
S_{i2}]$
with empty argument if
$M_i
\in
\{1,2\}$.
The latter indicates that the factors are elements of
$\epsp$
and
$\epsp2$
while the general case means that the factors are from
 $\epsp
M_{i[S_{i1},S_{i2}]}$.

Each element is the product of $N =
M_1
+
M_2$
atomic Euler pseudoprimes.

The elements of this set satisfy
(\ref{eq:ProofSSAssumption})
for all prime bases up to and including $a$,
but not the next prime after $a$.

If
$T_1
=
T_2$,
we will use
$\epsp
N_{[T_1]^2}(a)$
\end{definition}

See
Figure~\ref{fig:epsp7}
for an example of this definition,
where we show
$\epsp7_{[1,6[3[1,2]]^2]}(a)$.
Note that the two
$\epsp3_{[1,2]}(a)$
are created identically and thus we use a square.

\begin{figure}
\centering
\includegraphics[width=1.0\linewidth]{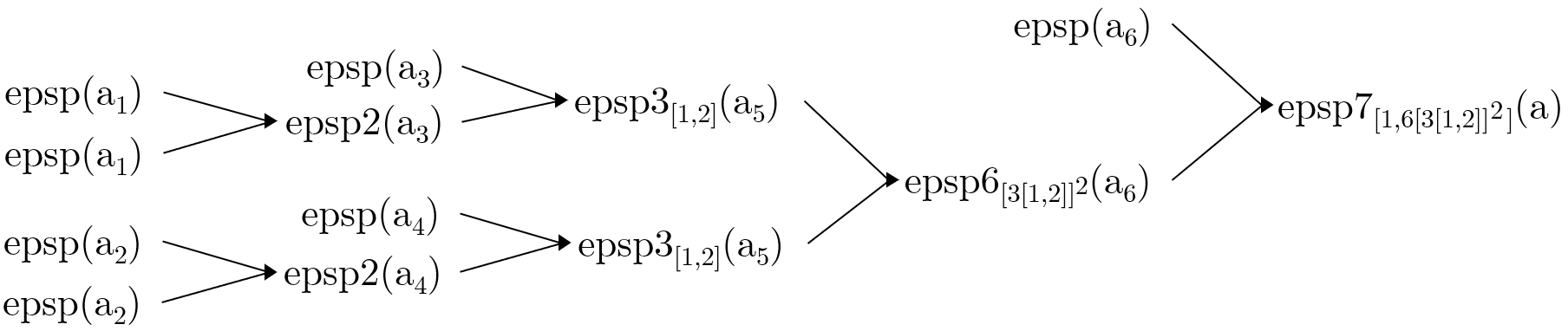}
\caption{Schematic way of how we create $\epsp 7_{[1,6[3[1,2]^2]]}(a)$.}
\label{fig:epsp7}
\end{figure}

Based on experiments we decided to prioritize symmetric constructions.
Hence,
we computed
$\epsp4_{[2]^2}(a')$
for some $a'
\geq
37$,
by multiplying two elements of
$\epsp2(a)$.
In
Figure~\ref{fig:epsp4}
this is schematically shown.
Note that we now consider fewer options than just multiplying 4 atomic Euler pseudoprimes
with each other,
but it does speed up the process and we expect to still find the good ones.
The size of each set can be found in the fifth column of
Table~\ref{tab:epsp_numbers}.
The best one found now is
$$\mbox{\tiny
12734329098687333118404752403795776326472689932237076660290997556917267532133629644627387201}$$
which survives up till 191,
the 43rd base.

\begin{figure}[h]
\centering
\includegraphics[width=0.6\linewidth]{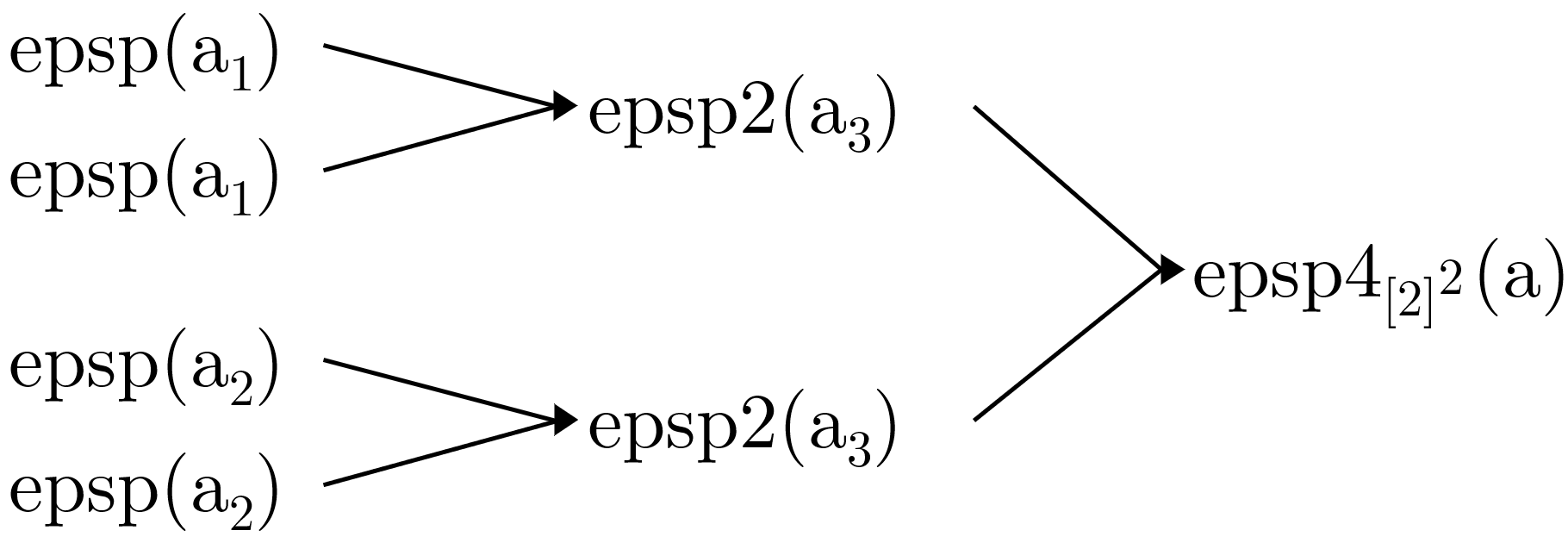}
\caption{Schematic way of how we create $\epsp4_{[2]^2}(a)$}
\label{fig:epsp4}
\end{figure}

Since we now have reached a base larger than 150,
before continuing with
$\epsp8_{[4[2]^2]^2}(a)$,
we filter our
$\epsp4_{[2]^2}(a)$
down to only the Euler pseudoprimes whose smallest factor is greater than 200.
Note that we also from now on will consider larger bases,
starting at 71.
This choice is simply due to the large number of elements we are already finding.
We show the results of these computations in
Table~\ref{tab:epsp_numbers_200}.
The best primes we now have are

\begin{spacing}{0.5}
\begin{align*}
&\mbox{\tiny
785106814369213084281283063908979192118832154301111712456158708772052111844251666955723549037826979}\\
 \vspace{-5pt}
&\mbox{\tiny
1202161408637296926808299884248207740798821548861208208589759167845256874008467201}
\end{align*}
\end{spacing}

{\doublespacing
and}
\begin{spacing}{0.5}
\begin{align*}
&\mbox{\tiny
157714255773897113668884293394951261202240932033830203712863104157081133800659688260341049302324585}\\
 \vspace{-5pt}
&\mbox{\tiny
73314351348527389031410555808199074136270115924340123308057997982961191422511059201}
\end{align*}
\end{spacing}

\smallskip

Both survive till 199,
the 46th prime base.

To continue on this,
we first further reduced the input to only look at Euler pseudoprimes with smallest
factor greater than 300 and then created
$\epsp16_{[8[4[2]^2]^2]^2}(a)$
using our filtered
$\epsp8_{[4[2]^2]^2}(a)$
as input.
The results can be found in
Table~\ref{tab:epsp_numbers_300}.
This is the first time we found fewer Euler pseudoprimes as products than we had inputs.
Disproportionally more products failed to be squarefree.
The main reason for this is that our starting inputs are atomic Carmichael numbers
$<10^{24}$
and Carmichael numbers have factors
$p_i$
so that the
$p_i
-1$ share many factors.
Thus the number of prime factors we are working with is also limited.
Cross-multiplying these numbers combined the factors more,
so at this point most of the numbers have factors in common,
i.e.,
their products will not be squarefree.
We unwittingly sped up this process by filtering to only numbers with smallest factor
larger than $300$.
A high percentage of numbers in
$\epsp16_{[8[4[2]^2]^2]^2}(a)$
have factor 313.
If we try to now create
$\epsp32_{[16[8[4[2]^2]^2]^2]^2}(a)$
we do not find any Euler pseudoprimes,
and all products that would be in class A fail due to not being squarefree.

So,
in a second attempt,
we went back to the cut off of 200 for the smallest factor and did some more searching
for new Euler pseudoprimes.
Even though we could not complete this search due to time,
we did find our best Euler pseudoprime that survives up to $211$,
the $47$th prime base.
It has 1230 bits.

\begin{spacing}{0.5}
\begin{align*}
&\mbox{\tiny
297180253242403184105341150390525401177755979450850674143574534366013877426704534593555915420699343}\\
 \vspace{-5pt}
&\mbox{\tiny
889059073145832891673993841274696136803836427921670679792296684633516565295275589915225338961827697}
\\
 \vspace{-5pt}
&\mbox{\tiny
395774510820672522818570055652099552788792514216529771842300258718279723455898329630840119845009750}
\\
 \vspace{-5pt}
&\mbox{\tiny
75999311171856228688510024383030748708219201214793184433340476692190457601}
\end{align*}
\end{spacing}

\begin{table}[ht]
\noindent\parbox{.44\linewidth}{
\vspace{0pt}
\centering
\captionsetup{width=6cm}
\begin{tabular}{R{0.67cm}|R{1.8cm}|R{2.2cm}}
$a$
&
$\epsp4_{[2]^2}(a)$
&
$\epsp8_{[4[2]^2]^2}(a)$
  \\
\hline
71
 &
211\,137
&
0
\\
73
 &
106\,652
&
18\,414\,583
\\
79
 &
53\,794
 &
13\,858\,342\\
83
 &
26\,930
 &
8\,092\,698
\\
89
 &
13\,494
 &
4\,342\,026
\\
97
 &
6\,770
  &
2\,247\,049
\\
101
&
3\,408
  &
1\,142\,207
\\
103
&
1\,660
  &
575\,054\\
107
&
843
     &
288\,552
\\
109
&
437
     &
144\,946
\\
113
&
193
     &
72\,518
\\
127
&
95
&
36\,431
\\
131
&
45
&
18\,342
\\
137
&
20
&
8\,920
\\
139
&
13
&
4\,491
\\
149
&
5
&
2\,263
\\
151
&
2
&
1\,085
\\
157
&
2
&
523
\\
163
&
2
&
 286
\\
167
&
0
&
 118
\\
173
&
2
&
77
\\
179
&
0
&
36
\\
181
&
0
&
14
\\
191
&
1
&
11
\\
193
&
0
&
 6
\\
197
&
0
&
4
\\
199
&
0
&
2
\\
\hline
Total
&
5\,872\,144
&
49\,250\,584
\\
\end{tabular}
\caption{Number of found class A Euler pseudoprimes (with smallest factor larger than 200). Note that  $\epsp4_{[2]^2}(a)$ is created from $\epsp2(a)$ for $a \geq 37$, while $\epsp8_{[4[2]^2]^2}$ is created with input of $a \geq 71$.}
\label{tab:epsp_numbers_200}
}
\hspace{0.5cm}
\parbox{.44\linewidth}{
\vspace{0pt}
\centering
\captionsetup{width=6cm}
\begin{tabular}{R{0.67cm}|R{2.2cm}|R{2cm}}
$a$
&
$\epsp8_{[4[2]^2]^2}(a)$
&
$\epsp16_{[8[4[2]^2]^2]^2}(a)$
  \\
\hline
83
 &
280\,229
&
0
\\
89
 &
149\,495&
233\,780\\
97
 &
78\,106&
183\,464
\\
101
&
40\,158
&
110\,642
\\
103
&
19\,766
&
59\,993\\
107
&
10\,030
&
31\,283
\\
109
&
5\,010
&
15\,810\\
113
&
2\,504
&
8\,059\\
127
&
1\,240&
3\,993\\
131
&
648
&
2048\\
137
&
325
&
1019\\
139
&
139
&
535
\\
149
&
84
 &
245\\
151
&
33
 &
137\\
157
&
20
 &
66\\
163
&
6
&
28\\
167
&
5
&
15\\
173
&
1
&
6\\
179
&
0
&
4\\
181
&
2
&
0\\
191
&
1
&
1\\
193
&
1
&
0\\
\hline
Total
&
1\,713\,060
&
651\,128
\\
\end{tabular}
\caption{Number of found class A Euler pseudoprimes (with smallest factor larger than 300). Note that  $\epsp4_{[2]^2}(a)$ is created from $\epsp2(a)$ for $a \geq 37$, while $\epsp8_{[4[2]^2]^2}$ is created with input of $a \geq 83$.}
\label{tab:epsp_numbers_300}
}
\end{table}

\relax
\section{Sophie Germinus pseudoprimes}\label{app:anonymous}
The motivation for this paper was to find integers with many Euler liars.
We have seen that Carmichael numbers are indeed the way to go for finding these.
In this section we will introduce the next best option for the search for Euler pseudoprimes,
a certain type of composite number we introduce as the Sophie Germinus
pseudoprimes\footnote{After
a suggestion by Daniel
J.\ Bernstein.}.
\relax
The only other composite numbers that we found in a search over all odd numbers to
$10^6$
were numbers of the following shape,
which we define as Sophie Germinus pseudoprimes.

\begin{definition}
Let
$p_1$
be prime and
$p_2
= 2
p_1
- 1$ also be prime.
Than we call
$p_1$
a Sophie Germinus prime and the product $n =
p_1p_2$
is a Sophie Germinus pseudoprime.
\end{definition}

\begin{lemma} \label{lem:germinus}
Let $n$ be a Sophie Germinus pseudoprime,
with
$p_1
>3$,
then there are
$\varphi(n)/4$
Euler liars in
$\mathbb{Z}_n^*$
and $-1$ and $1$ appear equally often.
\end{lemma}

\begin{proof}
Let $n =
p_1p_2$
with
$p_1
>
3$ and
$p_2
=
2p_1
- 1$,
with
$p_1,
p_2$
prime.
This implies that $
p_1
\equiv
1
\pmod{6}
$ and $
p_2
\equiv
1
\pmod{12}$.
For $ i = 1,
2 $ let $
g_i
$ be a primitive root modulo $
p_i
$,
and define $
A_i
\equiv
g_i
\pmod{p_i}
$ and $
A_i
\equiv
1
\pmod{p_j}
$ for $ j
\neq
i $.
Then $
\{
A_1,
A_2
\}
$ is a basis for $
\Z_n^*
$,
i.e.,
we can write any $a
\in
\mathbb{Z}_n^*$
as $a =
A_1^{e_1}A_2^{e_2}$.

Note that
$\lambda(n)
=
2(p_1
-1)$,
$\varphi(n)
= 2
(p_1
-
1)^2$,
$n -1 =
(p_1
- 1)(2
p_1
+ 1)$ and $2(n-1) =
(p_2
-
1)(p_2
+ 2)$.

Now since
$(p_1
-1)
|
(n-1)$,
we have $
A_1^{n-1}
\equiv
1
\pmod{p_1}$
and $
A_1^{n-1}
\equiv
1
\pmod{p_2}
$ by definition,
so $
A_1^{n-1}
\equiv
1
\pmod{n}
$.
But $
A_1^{(n-1)/2}
\not\equiv
\pm
1
\pmod{n}
$ because $
A_1^{(n-1)/2}
\equiv
\left(A_1^{(p_1-1)/2}\right)^{2p_1+1}
\equiv
(-1)^{2p_1+1}
= -1
\pmod{p_1}
$,
while $
A_1^{(n-1)/2}
\equiv
1
\pmod{p_2}
$.
Similarly,
$
A_2^{n-1}
\equiv
1
\pmod{p_1}
$,
but $
A_2^{n-1}
\equiv
\left(A_2^{(p_2-1)/2}\right)^{p_2+2}
\pmod{p_2}
\equiv
(-1)^{p_2+2}
= -1
\pmod{p_2}
$,
so $
A_2^{n-1}
\not\equiv
\pm
1
\pmod{n}$,
while $
A_2^{2(n-1)}
\equiv
1
\pmod{n}$
Therefore $
A_2^{(n-1)/2}
\not\equiv
\pm
1
\pmod{n}
$ is a fourth root of $ 1 $ modulo $ n $.
Let $a
\equiv
A_1^{e_1}A_2^{e_2}
\pmod{n}$,
then
$a^{(n-1)/2}
\equiv
A_1^{e_1(p_1-1)/2}A_2^{e_2(p_2-1)/4}
\pmod{n}$.
It follows that there are $ 8 $ possibilities for $
a^{(n-1)/2}
\pmod{n}
$,
depending on
$e_1
\pmod2$
and
$e_2\pmod4$.
All $ 8 $ occur the same number of times,
namely $
\varphi(n)/8
$.

To have $
A_1^{e_1}A_2^{e_2}\equiv
\pm
1
\pmod{n}$
we need that
$A_1^{e_1(p_1-1)/2}
\pmod{p_1}$
matches
$A_2^{e_2(p_2-1)/4}
\pmod{p_2}$,
hence the latter must be
$\pm
1$,
limiting
$e_2$
to $0$ and
$2\pmod4$.
These match at 1 for residues $(0,0)$ and at $-1$ for residues $(0,2)$,
which we combine into requring $2
e_1
+
e_2
\equiv
0
\pmod{4}$.

For the Jacobi symbol we have $
\left(\frac{A_1^{e_1}A_2^{e_2}}{n}\right)
=
\left(\frac{A_1^{e_1}}{p_1}\right)
\left(\frac{A_2^{e_2}}{p_1}\right)
\left(\frac{A_1^{e_1}}{p_2}\right)
\left(\frac{A_2^{e_2}}{p_2}\right) =
(-1)^{e_1}\cdot
1
\cdot
1
\cdot
(-1)^{e_2}
=
(-1)^{e_1}$
because
$e_2$
is even.
\relax
Hence,
if
$2e_1+e_2\equiv
0
\pmod4$
then $
a^{(n-1)/2}
\equiv
 \left(\frac{a}{n}\right)
\pmod{n}
$.
So 2 of the 8 cases describe Euler liars and both signs appear equally often,
while in the other 6 cases $
a^{(n-1)/2}
\not\equiv
\pm
1
\pmod{n}
$.
\end{proof}

We conjecture the following and leave this for future work.
\begin{conjecture}
If an odd squarefree composite integer $ n $ has
$\varphi(n)/2$
Euler liars in
$\mathbb{Z}_n^*$
then it is a Carmichael number,
and if it has
$\varphi(n)/4$
Euler liars in
$\mathbb{Z}_n^*$
then it is either a Carmichael number or a Sophie Germinus pseudoprime.
\end{conjecture}

For a generalization of Lemma
\ref{lem:germinus},
see
\url{https://math.deweger.net/eulerliars/}.

\relax

\relax
\section*{Acknowledgments}
We thank Andrew Granville and Carl Pomerance for useful discussions.

\bibliographystyle{plainurl}
\bibliography{abbrev0,crypto,references}

\end{document}